\documentclass[a4paper, 12pt]{amsart}
\usepackage[english]{babel}
\usepackage{amsfonts}
\usepackage{amssymb}
\usepackage[leqno]{amsmath}
\usepackage{amsthm}
\usepackage{amscd}
\usepackage{xypic}
\usepackage[curve]{xy}
\usepackage{epsfig}
\usepackage{enumerate}

\input xy
\xyoption{all}
\newtheorem{theorem}{Theorem}[section]
\newtheorem{proposition}[theorem]{Proposition}

\newtheorem{corollary}[theorem]{Corollary}

\newtheorem{example}[theorem]{Example}
\newtheorem{lemma}[theorem]{Lemma}
\newtheorem{remark}[theorem]{Remark}

\def\length{\mathrm{length}\hspace{0.1cm}}

 \DeclareMathOperator{\Hom}{Hom}

\def\Cohom{\mathrm{Cohom}}

\def\Ext{\mathrm{Ext}}
\def\Hom{\mathrm{Hom}}
\def\Irr{\mathrm{Irr}}
\def\inj{\mathrm{inj}}
\def\comp{\mathrm{comp}}
\def\rad{\mathrm{rad}}

\def\dim{\mathrm{dim}}

\def\End{\mathrm{End}}
\def\Soc{\mathrm{soc\hspace{0.1cm}}}
\def\Socle{\mathrm{soc}}

\def\C{\mathcal{C}}

\def\T{\mathcal{T}}

\def\M{\mathcal{M}}

\def\point#1{*+[o][F-]{\text{\scriptsize $#1$}}}

\title{Serial coalgebras and their valued Gabriel quivers}
\author{Jos\'{e} G\'{o}mez-Torrecillas}
\author{Gabriel Navarro}
\address{Department of Algebra\\
University of Granada \\
Avda. Fuentenueva s/n\\ E-18071\\ Granada\\ Spain}
\email{gomezj@ugr.es, gnavarro@ugr.es}
\thanks{
Research supported by Spanish MEC project MTM2004-01406, and
FEDER}

\begin{document}

\maketitle

\begin{abstract}
We study serial coalgebras by means of their valued Gabriel quivers.
In particular, Hom-computable and representation-directed coalgebras
are characterized. The Auslander-Reiten quiver of a serial coalgebra
is described. Finally, a version of Eisenbud-Griffith theorem is
proved, namely, every subcoalgebra of a prime, hereditary and strictly
quasi-finite coalgebra is serial.
\end{abstract}

\section*{Introduction}

A systematic study of serial coalgebras was initiated in
\cite{cuadra-serial}, where, in particular, it was shown that any
serial indecomposable coalgebra over an algebraically closed field
is Morita-Takeuchi equivalent to a subcoalgebra of a path
coalgebra of a quiver which is either a cycle or a chain (finite
or infinite) \cite[Theorem 2.10]{cuadra-serial}. In this paper, we
take advantage of the valued Gabriel quivers associated to a
coalgebra to characterize indecomposable serial coalgebras over
any field (Theorem \ref{shapeserial}). In conjunction with
localization techniques (see Section \ref{localizacion}), this
combinatorial tool allows to complete the study made in
\cite{cuadra-serial} in more remarkable aspects. Thus, in Section
\ref{hom-comp}, we characterize Hom-computable serial coalgebras
in the sense of \cite{simson07b} (Proposition
\ref{serialcomputable}), and representation-directed coalgebras
(Proposition  \ref{representationdirected}). Section
\ref{finitedimensional} is devoted to describe the
Auslander-Reiten quiver of the category of finite dimensional
(right) comodules of a serial coalgebra.

It was observed in \cite{cuadra-serial} that a consequence of
\cite[Corollary 3.2]{eisenbud1} is that the finite dual coalgebra of
a hereditary noetherian prime algebra over a field is serial. In
Section \ref{EisenbudGriffith} we reconsider this result of Eisenbud
and Griffith from the coalgebraic point of view: we prove, using the
results developed in the previous sections, that any subcoalgebra of a prime, hereditary
and strictly quasi-finite coalgebra is serial (Corollary
\ref{EGtheorem}).

Throughout we fix a field $K$ and we assume $C$ is a
$K$-\emph{coalgebra}.  We refer the reader to the books \cite{Abe},
\cite{montgomery} and \cite{sweedler} for notions and notations
about coalgebras. Unless otherwise stated, all $C$-comodules are
right $C$-comodules. It is well-known that $C$ has a decomposition,
as right $C$-comodule,
$$C_C=\bigoplus_{i\in I_C}
\hspace{0.1cm} E_i^{t_i} \hspace{0.1cm},$$
 where $\{E_i\}_{i\in
I_C}$ is a complete set of pairwise non-isomorphic
\emph{indecomposable injective} right $C$-comodules and $t_i$ is a
positive integer for any $i\in I_C$. This produces a decomposition
of the \emph{socle} of $C$ (the sum of all its simple subcomodules),
$\Soc C$, as follows:
$$ \Soc C=\bigoplus_{i\in I_C} \hspace{0.1cm}
S_i^{t_i} \hspace{0.1cm},$$
 where $\{S_i\}_{i\in I_C}$ is a complete
set of pairwise non-isomorphic \emph{simple} right $C$-comodules. It
is easy to prove that
$$t_i=\frac{\dim_K S_i}{\dim_K(\End_C(S_i))}$$ for any $i\in  I_C$,
see \cite{simson2}.

For any right $C$-comodule $M$, we denote by $\Soc M$ the socle of
$M$ and by $E(M)$ its \emph{injective envelope}. We assume that
$\Soc E_i=S_i$, for each $i\in I_C$, and consequently, $E(S_i)=E_i$.

Throughout we denote by $G_i$ the division $K$-algebra of
endomorphism $\End_C(S_i)$ for each $i\in I_C$. The coalgebra $C$ is
said to be \emph{basic} if $t_i=1$ for any $i\in I_C$, or,
equivalently, if $\dim_K S_i=\dim_KG_i$ for any $i\in I_C$, or,
equivalently, if $S_i$ is simple subcoalgebra of $C$ for any $i\in
I_C$, see for instance \cite{simson07}. In particular, $C$ is called
\emph{pointed} if $\dim_K S_i=1$ for any $i\in I_C$.

If the field $K$ is algebraically closed then $C$ is pointed if and
only if $C$ is basic (cf. \cite[Corollary 2.7]{simson2}).

Since every coalgebra is Morita-Takeuchi equivalent (that is, their
categories of comodules are equivalent) to a basic one (cf.
\cite{chin}), throughout we assume that $C$ is basic and there are
decompositions
\begin{equation}\label{decomp1}
C=\bigoplus_{i\in I_C} \hspace{0.1cm} E_i \hspace{0.5cm} \text{ and
} \hspace{0.5cm} \Soc C=\bigoplus_{i\in I_C} \hspace{0.1cm} S_i
\hspace{0.1cm},
\end{equation}
where $E_i\ncong E_j$ and $S_i\ncong S_j$ for $i\neq j$.
Symmetrically, there exists the left-hand version of all the facts
explained above. In particular, $C$ admits a decomposition as left
$C$-comodule
\begin{equation}\label{decomp2}
 {_C}C=\bigoplus_{i\in I_C} \hspace{0.1cm} F_i
\hspace{0.5cm} \text{ and } \hspace{0.5cm} \Soc C=\bigoplus_{i\in
I_C} \hspace{0.1cm} S'_i \hspace{0.1cm}.
\end{equation}

\begin{remark} Observe that $S_i=S'_i$ for any $i\in I_C$, since $C$
is basic, and therefore each simple (left or right) $C$-comodule is
a simple subcoalgebra. Nevertheless, the right injective envelope
$E_i$ and the left injective envelope $F_i$ of $S_i$ could be
different.
\end{remark}

We recall from \cite{cuadra-serial} that a right $C$-comodule $M$ is
said to be uniserial if its lattice of subcomodules is a chain. This
property can be characterized through the socle filtration, namely,
$M$ has a filtration
$$0\subset \Soc M\subset \Socle^2 M\subset \cdots \subset M$$
called the \emph{Loewy series}, where, for $n>1$, $\Socle^nM$ is the
unique subcomodule of $M$ satisfying that $\Socle^{n-1}M\subset
\Socle^nM$ and
$$\frac{\Socle^{n}M}{\Socle^{n-1}M}=\Socle\left (\frac{M}{\Socle^{n-1}M}\right ),$$
see \cite{green} and \cite{navarro} for some properties of the Loewy
series.

\begin{lemma}\cite{cuadra-serial} The following statements are
equivalent:
\begin{enumerate}[$(a)$]
\item $M$ is uniserial.
\item The Loewy series is a composition series.
\item Each finite dimensional subcomodule of $M$ is uniserial.
\end{enumerate}
\end{lemma}

The coalgebra $C$ is said to be right (left) serial if any
indecomposable injective right (left) $C$-comodule is uniserial. $C$
is called serial if it is both right and left serial.

Throughout we denote by $\M^C_f$, $\M^C_{qf}$ and $\M^C$ the
category of finite dimensional, quasi-finite and all right
$C$-comodules, respectively. Dually, ${^C}\M_f$, ${^C}\M_{qf}$ and
${^C}\M$ denote the corresponding categories of left $C$-comodules.

A full subcategory $\T$ of $\M^C$ is said to be \emph{dense} (or a
\emph{Serre class}) if each exact sequence
$$\xymatrix{0 \ar[r] &M_1 \ar[r] & M \ar[r] & M_2 \ar[r] & 0}$$ in
$\M^C$  satisfies that $M$ belongs to $\T$ if and only if $M_1$ and
$M_2$ belong to $\T$. Following \cite{gabriel} and \cite{popescu},
for any dense subcategory $\T$ of $\M^C$, there exists an abelian
category $\M^{C}/\T$ and an exact functor $T:\M^{C}\rightarrow
\M^{C}/\T$, such that $T(M)=0$ for each $M \in \T$, satisfying the
following universal property: for any exact functor
$F:\M^C\rightarrow \C$ such that $F(M)=0$ for each $M\in\T$, there
exists a unique functor $\overline{F}:\M^C/\T \rightarrow \C$
verifying that $F=\overline{F}T$. The category $\M^C/\T$ is called
the \emph{quotient category} of $\M^C$ with respect to $\T$, and $T$
is known as the \emph{quotient functor}.

Let now  $\T$ be a dense subcategory of the category $\M^C$, $\T$ is
said to be \emph{localizing} (cf. \cite{gabriel}) if the quotient
functor $T:\M^C\rightarrow \M^C/\T$ has a right adjoint functor $S$,
called the \emph{section functor}.  If the section functor is exact,
$\T$ is called \emph{perfect localizing}. Let us list some
properties of the localizing functors (cf. \cite[Chapter
III]{gabriel}).

\begin{lemma} Let $\T$ be a dense subcategory
of the category of right comodules $\M^C$ over a coalgebra $C$. The
following statements hold:
\begin{enumerate}[$(a)$]
\item $T$ is exact.
\item If $\T$ is localizing, then the section functor $S$ is left exact and
the equivalence $TS\simeq 1_{\M^C/\T}$ holds.
\end{enumerate}
\end{lemma}

From the general theory of localization in Grothendieck categories
\cite{gabriel}, it is well-known that there exists a one-to-one
correspondence between localizing subcategories of $\M^C$ and sets
of indecomposable injective right $C$-comodules, and, as a
consequence, sets of simple right $C$-comodules. More precisely, a
localizing subcategory is determined by an injective right
$C$-comodule $E=\oplus_{j\in J} E_j$, where $J\subseteq I_C$
(therefore the associated set of indecomposable injective comodules
is $\{E_j\}_{j\in J}$). Then $\M^C/\T\simeq \M^D$, where $D$ is the
coalgebra of coendomorphism $\Cohom_C(E,E)$ (cf. \cite{takeuchi} for
definitions), and the quotient and section functors are
$\Cohom_C(E,-)$ and $-\square_D E$, respectively.

In \cite{cuadra}, \cite{jmnr} and \cite{woodcock}, localizing
subcategories are described by means of idempotents in the dual
algebra $C^*$. In particular, it is proved that the quotient
category $\M^C/\T$ is the category of right comodules over the
coalgebra $eCe$, where $e\in C^*$ is an idempotent associated to the
localizing subcategory $\T$ (that is, $E=Ce$, where $E$ is the
injective right $C$-comodule associated to the localizing
subcategory $\T$). The coalgebra structure of $eCe$ (cf.
\cite{radford}) is given by
$$\Delta_{eCe} (exe)=\displaystyle \sum_{(x)} ex_{(1)}e\otimes
ex_{(2)}e \hspace{0.4cm}\text{ and }\hspace{0.4cm}
\epsilon_{eCe}(exe)=\epsilon_C(x)$$ for any $x\in C$, where
$\Delta_C(x)= \sum_{(x)} x_{(1)} \otimes x_{(2)}$ using the
sigma-notation of \cite{sweedler}. Throughout we denote by $\T_e$
the localizing subcategory associated to the idempotent $e$. For
completeness, we recall from \cite{cuadra} (see also \cite{jmnr})
the following description of the localizing functors. We recall
that, given an idempotent $e\in C^*$, for each right $C$-comodule
$M$, the vector space $eM$ is endowed with a structure of right
$eCe$-comodule given by
$$\rho_{eM}(ex)=\sum_{(x)} ex_{(1)}\otimes
ex_{(0)} e$$ where $\rho_{M}(x)=\sum_{(x)} x_{(1)}\otimes x_{(0)}$
using the sigma-notation of \cite{sweedler}.

\begin{lemma} Let $C$ be a coalgebra and $e$ be an idempotent in
$C^*$. Then the following statements hold:
\begin{enumerate}[$(a)$]
\item The quotient functor $T:\M^C\rightarrow \M^{eCe}$ is naturally
equivalent to the functor $e(-)$. $T$ is also naturally equivalent
to the cotensor functor $-\square_{C} eC$ and the $\Cohom$ functor
$T_e=\Cohom_C(Ce,-)$.
\item The section functor $S:\M^{eCe} \rightarrow \M^C$ is naturally
equivalent to the cotensor functor $S_e=-\square_{eCe} Ce$.
\item $\T_e$ is perfect localizing if and only if $Ce$ is
injective as right $eCe$-comodule.
\end{enumerate}
\end{lemma}

We refer the reader to \cite{jmn}, \cite{jmn2} and \cite{jmnr} for
basic definitions, notations and properties about quivers and path
coalgebras. The localization in categories of comodules over path
coalgebras is described in detail in \cite{jmnr}.

\section{The valued Gabriel quiver}

Associating a graphical structure to a certain mathematical object
is a very common strategy. Sometimes, it provides us a nice method
for replacing the object with a simpler one and improving our
intuition about its properties. In our case, when dealing with
representation theory of coalgebras, the quivers associated to a
coalgebra play a prominent r\^{o}le in order to study their
structure in depth. This section is devoted to analyze the shape of
the so-called valued Gabriel quiver of a serial coalgebra carrying
on with the results obtained in \cite{cuadra-serial}. Throughout we
assume that $C$ is a basic coalgebra with decompositions
(\ref{decomp1}) and (\ref{decomp2}). Following \cite{justus}, let us
recall the notion of right \emph{valued Gabriel quiver} $(Q_C,d_C)$
of the coalgebra $C$ as follows: the set of vertices of $(Q_C,d_C)$
is the set of simple right $C$-comodules $\{S_i\}_{i\in I_C}$, and
there exists a unique valued arrow
$$\xymatrix{ S_i \ar[rr]^-{(d'_{ij},d''_{ij})} & & S_j}$$ if and
only if $\Ext_C^1(S_i,S_j)\neq 0$ and,
$$\text{$d'_{ij}=\dim_{G_i} \Ext_C^1(S_i,S_j)$   and
$d''_{ij}=\dim_{G_j} \Ext_C^1(S_i,S_j)$},$$ as a right $G_i$-module
and as a left $G_j$-module, respectively.

The (non-valued) Gabriel quiver of $C$ is obtained by taking the
same set of vertices and the number of arrows from a vertex $S_i$ to
a vertex $S_j$ is given by the integer $\dim_{G_i}
\Ext^1_C(S_i,S_j)$, where $\Ext^1_C(S_i,S_j)$ is viewed as right
$G_i$-module. If $C$ is pointed (or $K$ is algebraically closed)
then it is isomorphic to the one used by Montgomery in
\cite{montgomery2} and Woodcock in \cite{woodcock} in order to prove
that $C$ is a subcoalgebra of the path coalgebra of its (non-valued)
Gabriel quiver.

In \cite{simson06}, the valued Gabriel quiver of $C$ is described
through the notion of irreducible morphisms between indecomposable
injective right $C$-comodules. Let us denote by $\inj^C$ (respect.
${^C}\inj$) the full subcategory of $\M^C$ (respect. ${^C}\M$)
formed by socle-finite (i.e., comodules whose socle is
finite-dimensional) injective right (respect. left) $C$-comodules.
Let $E$ and $E'$ be two comodules in $\inj^C$. A morphism
$f:E\rightarrow E'$ is said to be irreducible if $f$ is not an
isomorphism and given a factorization
$$\xymatrix{ E \ar[rr]^-{f} \ar[rd]_-{g} &    &   E'\\
&  Z   \ar[ru]_{h} &  }$$ of $f$, where $Z$ is in $\inj^C$, $g$ is
a section, or $h$ is a retraction. Analogously to the case of
finite-dimensional algebras, there it is proven that the set of
irreducible morphism $\Irr_C(E_i,E_j)$ between two indecomposable
injective right $C$-comodules $E_i$ and $E_j$ is isomorphic, as
$G_j$-$G_i$-bimodule, to the quotient
$\rad_C(E_i,E_j)/\rad_C^2(E_i,E_j)$. We recall that, for each two
indecomposable injective right $C$-comodules $E_i$ and $E_j$, the
\emph{radical} of $\Hom_C(E_i,E_j)$ is the $K$-subspace
$\rad_C(E_i,E_j)$ of $\Hom_C(E_i,E_j)$ generated by all
non-isomorphisms. Observe that if $i\neq j$, then
$\rad_C(E_i,E_j)=\Hom_C(E_i,E_j)$. The square of $\rad_C(E_i,E_j)$
is defined to be the $K$-subspace
$$\rad_C^2(E_i,E_j) \subseteq \rad_C(E_i,E_j) \subseteq \Hom_C(E_i,E_j)$$
generated by all composite homomorphisms of the form
$$\xymatrix{E_i \ar[r]^-{f} & E_k \ar[r]^-{g} & E_j,}$$ where
$f\in \rad_C(E_i,E_k)$ and $g\in \rad_C(E_k,E_j)$. The $m$th power
$\rad_C^m(E_i,E_j)$ of $\rad_C(E_i,E_j)$ is defined analogously, for
each $m>2$.

\begin{lemma}\cite[Theorem 2.3(a)]{simson06} Let $C$ be a basic
coalgebra and set $G_i=\End_C(S_i)$ for each
$i\in I_C$. There is an arrow
$$\xymatrix{ S_i \ar[rr]^-{(d'_{ij},d''_{ij})} & & S_j}$$
in the right valued Gabriel quiver $(Q_C,d_C)$ of $C$ if and only if
$\Irr_C(E_j,E_i)\neq 0$ and
$$\text{$d'_{ij}=\dim_{G_j} \Irr_C(E_j,E_i)$  \hspace{0.3cm} and \hspace{0.3cm}
$d''_{ij}=\dim_{G_i} \Irr_C(E_j,E_i)$},$$ as right $G_j$-module and
 as left $G_i$-module, respectively.
\end{lemma}

Let us see that right serial coalgebras are easy to distinguish
from its valued Gabriel quiver. The following lemma is not new, it
appears in \cite[Proposition 1.7]{cuadra-serial}; anyhow, for the
convenience of the reader, we give a new proof only by means of
``coalgebraic" arguments.

\begin{lemma}\label{zerosimple}
A basic coalgebra $C$ is right serial if and only if the right
$C$-comodule $\Socle^2 E/\Soc E$ is zero or simple for each
indecomposable injective right $C$-comodule $E$.
\end{lemma}
\begin{proof}
Let $E$ be an indecomposable injective right $C$-comodule. Let us
prove that the quotient $\Socle^i E/\Socle^{i-1} E$ is simple or
zero for any $i\geq 2$. We proceed by induction on $i$. The case
$i=2$ is a consequent of the hypothesis. Let us now assume that the
statement holds for some integer $k\geq 2$, that is,
$$\Socle \left (\frac{E}{\Socle^{k-1} E} \right )=\frac{\Socle^k
E}{\Socle^{k-1} E}$$ is a simple comodule (if it was zero, then
$E=\Socle^kE$ and the result would follow) and hence the right
injective envelope of $E/\Socle^{k-1}E$ is an indecomposable
injective right $C$-comodule $E'$. Therefore, by \cite[Lemma
1.4]{navarro},
$$\frac{\Socle^{k+1} E} {\Socle^k E}
\cong \frac{\frac{\Socle^{k+1} E}{\Socle^{k-1} E}}{\frac{\Socle^{k}
E}{\Socle^{k-1} E}} \cong \frac{\Socle^2 \left (
\frac{E}{\Socle^{k-1} E} \right )} {\Socle \left (
\frac{E}{\Socle^{k-1} E} \right )} \leq \frac{\Socle^2 E'}{\Soc
E'}$$ which is simple or zero by hypothesis. The converse
implication is trivial.
\end{proof}

\begin{proposition}\label{arrows} A basic coalgebra $C$ is right serial if and only
if each vertex $S_i$ of the right valued Gabriel quiver $(Q_C,d_C)$
is at most the sink of one arrow and, if such an arrow exists, it is
of the form
$$\xymatrix{ S_j \ar[rr]^{(1,d)} & & S_i},$$ for some vertex $S_j$ and some
 positive integer $d$. In
particular, if $C$ is pointed, $C$ is right serial if and only if
each vertex in the (non-valued) Gabriel quiver of $C$ is the sink of
at most one arrow.
\end{proposition}
\begin{proof}
Recall that, for any simple right $C$-comodule $S_i$,
 $$\Ext^1_C(S_j,S_i)\cong \Hom_C(S_j,E_i/S_i)$$ as right $G_j$-modules for all
simple right $C$-comodule $S_j$, see for instance \cite{justus} and
\cite[Lemma 1.2]{navarro}.

Assume now that $C$ is right serial and $E_i/S_i\neq 0$ (otherwise
there is no arrow ending at $S_i$) then
 $E_i/S_i$ is a subcomodule of an indecomposable injective right comodule $E_j$, and then
$$\Ext^1_C(S,S_i)\cong \Hom_C(S,E_i/S_i)\cong \left\{
                    \begin{array}{ll}
                      G_j, & \text{if $S_j=S$}; \\
                      0, & \text{otherwise.}
                    \end{array}
                  \right.$$
as right $\End_C(S)$-modules. Hence, there is a unique arrow ending
at $S_i$ of the form $$\xymatrix{ S_j \ar[rr]^{(1,d)} & & S_i}.$$

Conversely, the immediate predecessors of $S_i$ correspond to the
simple comodules contained in $\Soc (E_i/S_i)$. Since, by
hypothesis, there is only one arrow ending at $S_i$,
$\Soc(E_i/S_i)=(S_j)^{t}$ for some simple right comodule $S_j$ and
some positive integer $t$. Now, since $t$ is the first component of
the label of the arrow,
 $\Socle^2 E_i/\Soc E_i=\Soc(E_i/S_i)=S_j$ is a simple comodule. By the previous lemma, $C$ is right serial.
\end{proof}

Symmetrically, we prove that $C$ is left serial if and only if each
vertex $S$ of the left valued Gabriel quiver $(_CQ,{_C}d)$ is at
most the sink of one arrow and, if such an arrow exists, it is of
the form
$$\xymatrix{ S' \ar[rr]^{(1,d)} & & S},$$ for some vertex $S'$ and some positive integer $d$.

The following simple result is very useful, see also \cite[Corollary
2.26]{justus2}.

\begin{proposition}\label{Gabvalued} Let $C$ be a basic coalgebra.
 The right valued Gabriel quiver $(Q_C,d_C)$ of $C$ is the
opposite valued quiver of the left valued Gabriel quiver
$(_CQ,{_C}d)$ of $C$. Consequently, the left (non-valued) Gabriel
quiver of  $C$ is the opposite of the right (non-valued) Gabriel
quiver of $C$.
\end{proposition}
\begin{proof}
We recall from \cite{chin2} that there exists a duality
$D:\inj^C\rightarrow {^C}\inj$ given by $D=\Cohom_C(-,C)$ whose
inverse (which we also denote by $D$) is given by
$D=\Cohom_{C^{op}}(-,C^{op})$. Let us denote $D(E_i)=F_i$ (and then
$D(F_i)=E_i$) for each indecomposable injective right $C$-comodule
$E_i$. Therefore, $F_i$ is an indecomposable injective left
$C$-comodule. Moreover, following \cite{cuadra-serial}, if $S_i$ is
the socle of $E_i$ and $E_i=Ce_i$ for some idempotent $e_i\in C^*$,
$$F_i=D(E_i)=\Cohom_C(Ce_i,C)=e_iC,$$
then $\Soc F_i=S_i$.  Summarizing, $E_i$ and $F_i$ are the right and
the left injective envelopes of $S_i$, respectively.

Now, since $\Hom_C(E_i,E_j)\cong \Hom_C(F_j,F_i)$, for each two
indecomposable injective right comodules $E_i$ and $E_j$, it is easy
to see that also $\Irr_C(E_i,E_j)\cong \Irr_C(F_j,F_i)$ and then
$\dim_K\Irr_C(E_i,E_j)=\dim_K \Irr_C(F_j,F_i)$. Thus
$$\begin{array}{rl} \dim_{G_i} \Irr_C(E_i,E_j)& =\displaystyle\frac{\dim_K
\Irr_C(E_i,E_j)}{\dim_K S_i}\\ & =\displaystyle\frac{\dim_K
\Irr_C(F_j,F_i)}{\dim_K S_i}\\ & =\dim_{G_i} \Irr_C(F_j,F_i)
\end{array}$$
Analogously, $\dim_{G_j} \Irr_C(E_i,E_j)=\dim_{G_j}
\Irr_C(F_j,F_i)$. Therefore, there exists an arrow
$$\xymatrix{ S_i\ar[rr]^{(d'_{ij},d''_{ij})} & & S_j}$$ in $(Q_C,d_C)$
if and only if there exists an arrow
$$\xymatrix{ S_j\ar[rr]^{(d''_{ij},d'_{ij})} & & S_i}$$ in
$({_C}Q,{_C}d)$ and the result follows.

\end{proof}

Let us now prove the main result of this section that generalizes
 \cite[Theorem 2.10]{cuadra-serial}. In what
 follows we denote the labeled arrows $\xymatrix@1{\circ \ar[r]^-{(1,1)}&
 \circ}$ simply by $\xymatrix@1{\circ \ar[r]&
 \circ}$. As well we denote a valued quiver $(Q,d)$ simply by $Q$ if
 $(d^1_{ij},d^2_{ij})=(1,1)$ for any $i$ and $j$.

\begin{theorem}\label{shapeserial}
Let $C$ be a indecomposable basic coalgebra over an arbitrary field
$K$. Then $C$ is serial if and only if the right (and then also the
left) valued Gabriel quiver of $C$ is one of the following valued
quivers:

$ \begin{array}[b]{l}

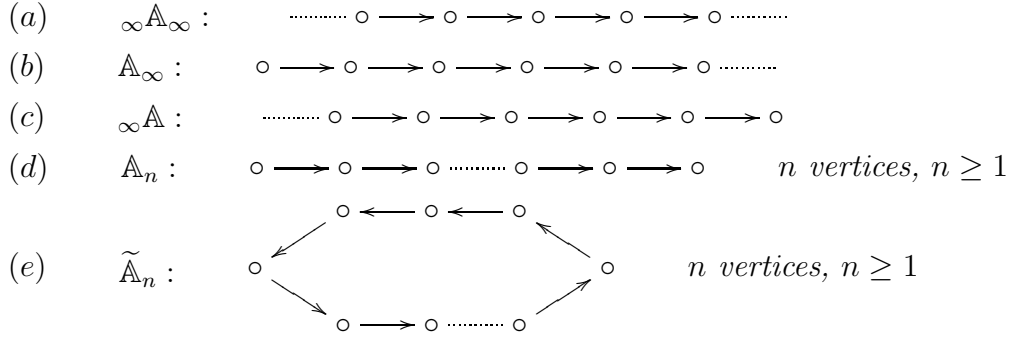
\xymatrix@C=20pt{(a) & _{\infty}\mathbb{A}_\infty: & \ar@{.}[r] &
\circ \ar[r] & \circ \ar[r] & \circ \ar[r] & \circ \ar[r] & \circ
\ar@{.}[r] & }

\\

\xymatrix@C=20pt{(b) & \mathbb{A}_\infty: & \circ\ar[r] & \circ
\ar[r] & \circ \ar[r] & \circ \ar[r] & \circ \ar[r] & \circ
\ar@{.}[r] & }

\\

\xymatrix@C=20pt{(c) &  _{\infty}\mathbb{A}: & \ar@{.}[r] & \circ
\ar[r] & \circ \ar[r]& \circ \ar[r] & \circ \ar[r] & \circ \ar[r] &
\circ}

\\

\xymatrix@C=20pt{(d) & \mathbb{A}_n: &  \circ \ar[r] &  \circ \ar[r]
&  \circ \ar@{.}[r] & \circ \ar[r] & \circ \ar[r] & \circ &
\text{$n$ vertices, $n\geq 1$}}

\\

\xymatrix@C=20pt@R=5pt{  &   &   & \circ \ar[ld]  & \circ \ar[l] &   \circ \ar[l]&   &  \\
(e)  & \widetilde{\mathbb{A}}_n
:  &\circ \ar[rd] &  &    &   & \circ \ar[lu]  & \text{$n$ vertices, $n\geq 1$} \\
 &  &       & \circ \ar[r] &\circ \ar@{.}[r] & \circ \ar[ru] &  &
 }
 \end{array}$
\end{theorem}

\begin{proof}
Assume that $C$ is serial and let $S$ be a simple right (and left)
$C$-comodule. Since $C$ is right serial, there exists at most one
arrow in $(Q_C,d_C)$ ending at $S$. Analogously, since $C$ is left
serial there exists at most one arrow in $(_CQ,{_C}d)$ ending at
$S$ and then, by Proposition \ref{Gabvalued}, we may deduce that
there is at most one arrow in $(Q_C,d_C)$ starting at $S$. Also,
by Proposition \ref{arrows} and its left-hand version and
Proposition \ref{Gabvalued}, any of these possible arrows are
labelled by $(1,1)$.

Taking into account the above discussion, let $S_0$ be a simple
right comodule. If there is no arrow neither ending nor starting at
$S_0$, i.e., $S_0$ is an isolated vertex, since $C$ is
indecomposable (and therefore $Q_C$ is connected, cf.
\cite{simson06}) then $Q_C=\mathbb{A}_1$. Similarly, if there is a
loop at $S_0$, then $Q_C=\widetilde{\mathbb{A}}_1$. Therefore we may
assume that there is no loop in $Q_C$ and then $S_0$ is inside a
(maybe infinite) path
$$\xymatrix@C=15pt{  \ar@{.}[r] & \ar[r] &S_{-n}  \ar[r] & \ar@{.}[r] &
\ar[r]& S_{-1} \ar[r]& S_{0} \ar[r] & S_1 \ar[r] & \ar@{.}[r] &
\ar[r] & S_m \ar[r] & \ar@{.}[r] & }$$ If there exist two
non-negative integers $n$ and $m$ such that $S_{-n}=S_m$, then $Q_C$
must be a crown, i.e., $Q_C=\widetilde{\mathbb{A}}_p$ for some
integer $p$. If not, $Q_C$ must be a line, that is, it is a quiver
as showed in ($a$) ,($b$), ($c$) or ($d$) depending on the
finiteness of the two branches. Clearly, the converse holds.
\end{proof}
\begin{corollary}
A basic coalgebra $C$ is serial if and only if each of the connected
component of its right (or left) valued Gabriel quiver is either
$_{\infty}\mathbb{A}_\infty$, or $\mathbb{A}_\infty$, or
$_{\infty}\mathbb{A}$; or $\mathbb{A}_n$ or $
\widetilde{\mathbb{A}}_n$ for some $n\geq 1$.
\end{corollary}

The problem of describing explicitly right (and not left) serial
coalgebras by means of its valued Gabriel quiver turns out much more
difficult. Using  a reasoning similar to the proof of Theorem
\ref{shapeserial}, we now state, without proof, an approximation to
this question. By a  diagram
$$ 
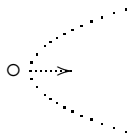
\xymatrix@C=15pt@R=15pt{  &    &  \ar@{.}@/_41pt/[dd]  \\
              \circ \ar@{.>}[r]&    &      \\
              &   &         }
              $$
we mean a quiver which is a subquiver of an (infinite) tree with
descendent orientation, that is, a subquiver of a quiver of the
following shape:
$$
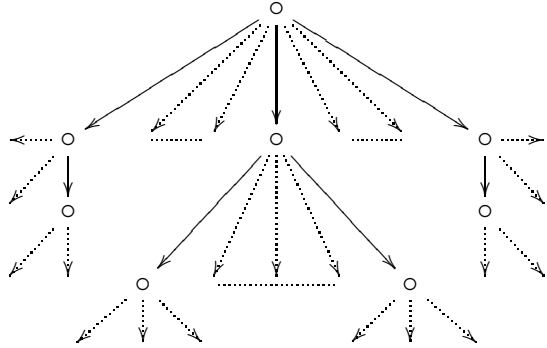
\xymatrix@R=15pt@C=15pt{  &   &  &  & \circ \ar[rrrdd]\ar@{.>}[rrdd]
\ar@{.>}[rdd]\ar@{.>}[ldd]\ar@{.>}[lldd]   \ar[llldd] \ar[dd] &  & &   &  \\
   &   &   & &  & &    &   &   &  & \\
    &  \circ \ar@{.>}[l] \ar[d] \ar@{.>}[dl]   & \ar@{.}[r] &  &  \circ \ar[rrdd]\ar@{.>}[rdd]
     \ar@{.>}[dd] \ar@{.>}[ldd]\ar[lldd] &  \ar@{.}[r]&   &
  \circ \ar@{.>}[r] \ar[d] \ar@{.>}[dr]  & \\
     &  \circ \ar@{.>}[ld] \ar@{.>}[d]  &        &    &        &
     &  & \circ \ar@{.>}[rd] \ar@{.>}[d]   & \\
     &      &   \circ \ar@{.>}[rd] \ar@{.>}[d] \ar@{.>}[ld] & \ar@{.}[rr]  &   &
     &  \circ \ar@{.>}[rd] \ar@{.>}[d] \ar@{.>}[ld]&
     & \\
     &   &      &    &   &  &     &   &
    }$$

\begin{proposition}
A basic coalgebra $C$ is right serial if and only if its right
valued Gabriel quiver is one of the following:
\begin{enumerate}[$(a)$]
\item If $(Q_C,d_C)$ is acyclic, then $Q_C$ has the form

$$
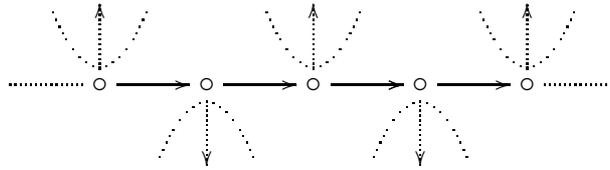
\xymatrix@C10pt{
  & \ar@{.}@/_27pt/[rr] &       &  &             & \ar@{.}@/_27pt/[rr]  &
       & &       &\ar@{.}@/_27pt/[rr]
   &     & & \\
 \ar@{.}[rr]& &  \circ \ar[rr]\ar@{.>}[u] & &
  \circ \ar[rr] \ar@{.>}[d]& & \circ\ar@{.>}[u]\ar[rr] \ar[rr]& & \circ
\ar[rr] \ar@{.>}[d]& & \circ \ar@{.>}[u]
\ar@{.}[rr]&  & \\
 &  &       &\ar@{.}@/^27pt/[rr]   &             &  &            & \ar@{.}@/^27pt/[rr] &
 & &   &   &  }$$
where each branch of the line may be finite or infinite.
  \item If $(Q_C,d_C)$ is not acyclic, then there exists a unique
  cycle in it, and $Q_C$ is of the form
$$
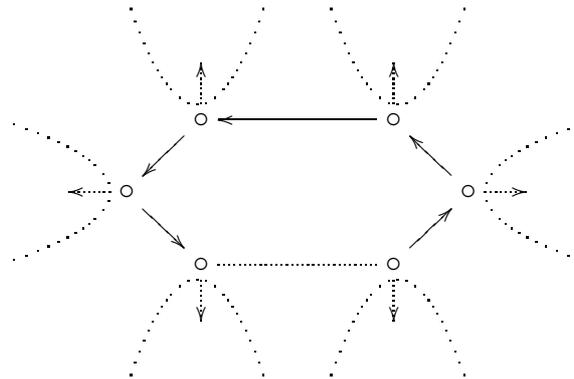
\xymatrix@C=15pt@R=15pt{
& & \ar@{.}@/_41pt/[rr]  &                 &   & \ar@{.}@/_41pt/[rr] &          & &  & \\
& &                      &                 &     & & & &  &\\
\ar@{.}@/^41pt/[dd]& & & \circ \ar[ld]\ar@{.>}[u]  &   & &   \circ \ar[lll]\ar@{.>}[u]& &  &  \ar@{.}@/_41pt/[dd]\\
& & \circ \ar[rd] \ar@{.>}[l]  &  & &   &   & \circ \ar[lu] \ar@{.>}[r]& & \\
& &    & \circ \ar@{.}[rrr] \ar@{.>}[d] & & & \circ \ar[ru] \ar@{.>}[d] & & &\\
& &                      &                 &     & & & &  &\\
& &   \ar@{.}@/^41pt/[rr] &                 &  &\ar@{.}@/^41pt/[rr]
& & & &
 }$$
\end{enumerate}
Moreover, each arrow $S_i\rightarrow S_j$ in $(Q_C,d_C)$ is
labelled by $(1,d_{ij})$, where $d_{ij}$ is a positive integer for
any $i$ and $j$ in $I_C$.
\end{proposition}

\section{Localization in serial coalgebras}\label{localizacion}

Let us now apply the localization techniques developed in
\cite{jmn2}, \cite{jmnr}, \cite{navarro} and \cite{simson07} to
(right) serial coalgebras. In particular, we give a characterization
of serial coalgebras by means of its ``local structure'', that is,
by means of
 its localized coalgebras.

The following proposition shows that the localization process
preserves the uniseriality of comodules and the seriality of
coalgebras.
 For each idempotent $e\in C^*$ we denote
 by $T_e$ the quotient functor $\Cohom_C(Ce,-):\M^C\rightarrow
 \M^{eCe}$.

\begin{proposition}\label{loc} Let $E=Ce$ be a quasi-finite injective
right $C$-comodule and $M$ a uniserial right $C$-comodule.
Then $T_e(M)=\Cohom_C(E,M)=eM$ is a uniserial right $eCe$-comodule.
\end{proposition}
\begin{proof}
Let us consider the (composition) Loewy series of $M$ as right
$C$-comodule,
$$\Soc M=S_1\subset \Socle^2 M\subset \Socle^3 M\subset \cdots \subset M$$
whose composition factors are $S_1$ and $S_k=M[k]=\Socle^k
M/\Socle^{k-1} M$ for $k\geq 2$. Since a simple $C$-comodule is
either torsion or torsion-free, let us suppose that $S_{i_1}$,
$S_{i_2}$, $S_{i_3}, \ldots$ are the torsion-free composition
factors of $M$, where $i_1< i_2< i_3<\cdots$.

For each $k<i_1$, $T_e(M[k])=T_e(S_k)=0$ and then
$T_e(\Socle^kM)=T_e(S_1)=0$. As a consequence, by \cite[Remark
2.3]{jmnr}, $T_e(\Socle^{i_1}M)=T_e(S_{i_1})=S_{i_1}$. Moreover,
since $M[i_1]=\Socle(M/\Socle^{i_1-1}M)=S_{i_1}$, then
$M/\Socle^{i_1-1}M$ is torsion-free and, by \cite[Proposition
3.2(c)]{navarro},
$$S_{i_1}=T_e(S_{i_1})=T_e \left (\displaystyle\Socle \left (\frac{M}{\Socle^{i_1-1}M} \right ) \right )=
\displaystyle \Socle \left ( T_e\left (\frac{M}{\Socle^{i_1-1}M}
\right ) \right )=\Soc T_e(M).$$ Applying the same arguments, we
may obtain that $T_e(\Socle^k M)=S_{i_1}$ for each $i_1\leq k <
i_2$, and $M/\Socle^{i_2-1}M$ is a torsion-free right
$C$-comodule. Then
\begin{multline*} T_e(M)[2]=\displaystyle\frac{\Socle^2T_e(M)}{\Soc T_e(M)}=\Socle \left (
\displaystyle \frac{T_e(M)}{\Soc T_e(M)} \right )  =\Socle \left (
\displaystyle \frac{T_e(M)}{T_e(\Socle^{i_1}M)} \right ) \\ =\Socle
\left ( \displaystyle \frac{T_e(M)}{T_e(\Socle^{i_2-1}M)} \right )=
T_e \left (\displaystyle\Socle \left (\frac{M}{\Socle^{i_2-1}M}
\right ) \right )=S_{i_2}
\end{multline*}
Thus $\Socle^2 T_e(M)=T_e(\Socle^{i_2} M)$. If we continue in this
fashion, we may prove that
$$T_e(\Socle^{i_1} M)\subset T_e(\Socle^{i_2}M)\subset T_e( \Socle^{i_3}M)\subset \cdots \subset T_e(M)$$
is the Loewy series of $T_e(M)$. Hence $T_e(M)$ is uniserial as a
right $eCe$-comodule.
\end{proof}

\begin{corollary}\label{loc1} Let $C$ be a right (left) serial coalgebra and $e\in C^*$ an idempotent. Then
the localized coalgebra $eCe$ is right (left) serial.
\end{corollary}
\begin{proof}
Let $\overline{E}_i$ be an indecomposable injective right
$eCe$-comodule. By \cite[Proposition 3.2]{navarro},
$T_e(E_i)=\overline{E}_i$, where $E_i$ is the indecomposable
injective right $C$-comodule such that
 $\Soc E_i=\Soc \overline{E}_i$. Since $E_i$ is uniserial, by Proposition \ref{loc}, so is $\overline{E}_i$.
\end{proof}

\begin{lemma}\label{lemmaloc2} Let $C$ be a coalgebra.
If the localized coalgebra $eCe$ is right (left) serial for each
idempotent $e\in C^*$ associated to a subset of simple comodules
with cardinal less or equal than three, then $C$ is right (left)
serial.
\end{lemma}
\begin{proof}
Let us suppose that $C$ is not right serial. By Lemma
\ref{zerosimple} there exists an indecomposable injective right
$C$-comodule $E$ such that $S_1\oplus S_2 \subseteq \Socle(E/S)$,
where $S_1$ and $S_2$ are simple comodules. Consider the
idempotent $e\in C^*$ associated to the set $\{S,S_1,S_2\}$. Then,
by \cite[Lemma 2.1]{navarro},
$$T_e(S_1\oplus S_2)=S_1\oplus S_2 \subseteq
T_e(\Socle(E/S))\subseteq \Socle(T(E/S))=\Socle(\overline{E}/S),$$
where $T_e(E)=\overline{E}$ is an indecomposable injective
$eCe$-comodule. Thus $eCe$ is not right serial.
\end{proof}
\begin{proposition}\label{equivserial} Let $C$ be a coalgebra. $C$ is right (left) serial if and
only if each socle-finite localized coalgebra of $C$ is right (left)
serial.
\end{proposition}
\begin{proof} Apply Corollary \ref{loc1} and Lemma \ref{lemmaloc2}.
\end{proof}

\begin{remark} The subsets of simple comodules mentioned in Lemma \ref{lemmaloc2} cannot have cardinal bounded by less than three.
For instance, if $C$ is the path coalgebra $KQ$ of the quiver $Q$
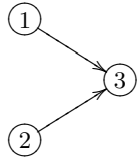
$$\xymatrix@R10pt{ \point{1} \ar[rd] & \\
  &   \point{3}\\ \point{2} \ar[ru] }$$
  then each localized coalgebra of $C$ at subsets of two or one simple comodules is right (and left) serial
  but clearly $C$ is not.
  \end{remark}

The former results are quite surprising since, as the following
proposition shows, the localization process increases the label of
an arrow (if exists) between two
torsion-free vertices.

\begin{proposition}\label{proplabels}
Let $C$ be a coalgebra and $e\in C^*$ idempotent. Let $S_1$ and
$S_2$ be two torsion-free simple $C$-comodules in the torsion
theory associated to the localizing subcategory $\T_e$. If there
exists an arrow $S_1\rightarrow S_2$ in $(Q_C,d_C)$ labelled by
$(d'_{12},d''_{12})$, then there exists an arrow $S_1\rightarrow
S_2$ in $(Q_{eCe},d_{eCe})$ labelled by $(t'_{12},t''_{12})$,
where $t'_{12}\geq d'_{12}$ and $t''_{12}\geq d''_{12}$.
\end{proposition}
\begin{proof}
Let us suppose that $\Socle(E_2/S_2)=\oplus_{i\in I_C} S_i^{n_i}$
for some non-negative integers $n_i$. Then there exists an arrow
$S_i\rightarrow S_2$ if and only if $n_i\neq 0$ and, furthermore,
in such a case, it is labelled by $(n_i,m_i)$ for some positive
integer $m_i$. Now, if there exists an arrow $S_1\rightarrow S_2$
in $(Q_C,d_C)$ labelled by $(d'_{12},d''_{12})$, then
$$S_1^{n_1}\subseteq \oplus_{i\in I_e}
S_i^{n_i}=T_e(\Socle(E_2/S_2))\subseteq \Socle T(E_2/S_2)\subseteq
\overline{E}_2/S_2,$$ since $S_1$ and $S_2$ are torsion-free. Hence
there is an arrow $$\xymatrix{S_1 \ar[rr]^-{(t'_{12},t''_{12})} &
&S_2}$$ in $(Q_{eCe},d_{eCe})$ and $t'_{12}=\dim_{G_1}
\Hom_{eCe}(S_1,\overline{E}_2/S_2)\geq n_1=d'_{12}$. By the
left-hand
 version of this reasoning and Proposition \ref{Gabvalued}, also
 $t''_{12}\geq d''_{12}$.
\end{proof}
\begin{remark} It is not possible to prove the equalities on the
components of the labels in the statement of Proposition
\ref{proplabels}. The reader only have to consider the path
coalgebra $KQ$ of the quiver $Q$
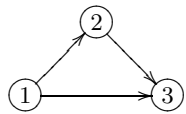
$$\xymatrix@R=15pt@C=15pt{ & \point{2} \ar[rd] & \\ \point{1} \ar[rr] \ar[ru] & &
\point{3}}$$ and the idempotent $e\in (KQ)^*$ associated to the
subset $\{1, 3\}$.
\end{remark}

\section{Hom-computable and representation-directed serial
coalgebras}\label{hom-comp}

 The study of the directing modules of an artin algebra comes from different motivations.
 On the one hand, they are treated as a generalization of the modules lying in a postprojective
 or a preinjective component (or more generally,
 in an acyclic component) of the Auslander-Reiten quiver of this algebra. Hence,
 they possess common properties with these modules as, for instance, they are determined up to isomorphism
 by their composition factors. On the other hand, they have interesting properties of their own as,
 for example, that any algebra having a sincere and directing
  module is a tilted algebra (that is, the endomorphism algebra
   of a hereditary algebra). It is also well-known that a
    representation-directed algebra (all its modules are directing)
     is finite representation-type, see \cite{simsonbook1} for details.
      This section deals with representation-directed coalgebras as
       defined in \cite{simson07b}. In particular, we describe the
        representation-directed serial coalgebras following the ideas
         of the previous sections, i.e., by means of their valued Gabriel
          quiver and using the localization in categories of comodules.
           In order to do this, we shall make use of the so-called
            computable comodules and Hom-computable coalgebras.

Assume that $C$ is a basic coalgebra with fixed decompositions
(\ref{decomp1}) and (\ref{decomp2}). Following \cite{simson07b}, a
right $C$-comodule $M$ is defined to be \emph{computable} if, for
each $i\in I_C$, the sum $\ell_i(M)=\sum_{n=1}^{\infty}
\ell_i(M[n])$, called the composition $S_i$-length of $M$, is
finite, where $M[n]=\Socle^n M/\Socle^{n-1}M$ and $\ell_i(M[n])$ is
the number of times that the simple comodule $S_i$ appears as a
summand in a semisimple decomposition of $M[n]$. We denote by
$\comp^C$ (resp. by $^C\comp$) the full subcategory of $\M^C$ (resp.
of $^C\M$) whose objects are computable comodules. The coalgebra $C$
is said to be \emph{Hom-computable} if every indecomposable
injective right $C$-comodule is computable or, equivalently (cf.
\cite{simson07b}), if $\Hom_C(E_i,E_j)$ has finite $K$-dimension for
any two indecomposable injective right $C$-comodules $E_i$ and
$E_j$. Therefore, by the duality $D:\inj^C\rightarrow {^C}\inj$
stated in \cite{chin2}, the notion of Hom-computability is
left-right symmetric. We now describe Hom-computable serial
coalgebras. For that purpose we give a version for coalgebras of the
Periodicity Theorem proved by Eisenbud and Griffith in
\cite{eisenbud1}.

\begin{lemma}\label{lemilla} Let $C$ be a coalgebra and $N$ a uniserial right $C$-comodule.
 If $M$ is a subcomodule of $N$ then
$M$ is uniserial and, moreover, $\Socle^t M=\Socle^t N$ for any
positive integer $t$ such that $\Socle^t M\neq \Socle^{t-1} M$. As a
consequence, $M$ is uniserial if and only if every subcomodule of
$M$ is uniserial.
\end{lemma}
\begin{proof}
Obviously, if $M\leq N$ then $\Soc M=\Soc N=S$ is a simple comodule.
Let us now assume that $\Socle^k M=\Socle^k N$ for each $k\leq t-1$,
and also $\Socle^{t-1} M\neq \Socle^t M$. Then
$$0\neq \displaystyle\frac{\Socle^t M}{\Socle^{t-1}M}\leq
\frac{\Socle^t N}{\Socle^{t-1} N}\cong S_t,$$ where $S_t$ is a
simple comodule. That is, $\Socle^t M/\Socle^{t-1} M\cong S_t$. Thus
$M$ is uniserial and, by its definition, $\Socle^t M=\Socle^t N$.
\end{proof}

\begin{proposition}[Periodicity Theorem] Let $C$ be an indecomposable serial
coalgebra and $E_0$ an indecomposable injective right $C$-comodule.
Let the sequence of composition factors of $E_0$ be $S_0=E_0[1]=\Soc
E$, $S_1=E_0[2]$, $S_2=E_0[3], \ldots$. Suppose that $S_1\cong S_k$
for some $k>1$, and let $h\neq 1$ be the smallest such integer. Then
the valued Gabriel quiver of $C$ is $\widetilde{\mathbb{A}}_{h}$,
and $S_m\cong S_n$ if and only if $m\equiv n (\text{mod $h$})$. If
there is no such an $h$, then $S_m\cong S_n$ implies $m=n$.
\end{proposition}
\begin{proof} Let us assume that  $S_0\ncong S_h$ for any $h>1$.
Suppose also that $S_n\cong S_m$ for some $m\neq n$ and,
furthermore, this is the first repetition, i.e., $S_i\ncong S_j$ for
$i,j<m$. Let us consider the injective right $C$-comodule
$E=E_0\oplus E_n=Ce$, where $e=e_0+e_n$. Then, by Proposition
\ref{loc} and its proof, $eCe$ is serial, $T_e(E_0)=\overline{E}_0$
is uniserial and its Loewy series is
$$S_0\subseteq T_e(\Socle^n E_0)\subseteq T_e(\Socle^m E_0)\subseteq
\cdots \subseteq \overline{E}_0,$$ where $\Socle^2
\overline{E}_0/S_0\cong S_n$ and $\Socle^3 \overline{E}_0/\Socle^2
\overline{E}_0\cong S_n$. Let now $M=\overline{E}_0/\Socle^2
\overline{E}_0\leq \overline{E}_n/S_n$. Then
$$S_n\cong\Socle^3 \overline{E}_0/\Socle^2 \overline{E}_0=\Soc
M\subseteq \Socle(\overline{E}_n/S_n).$$ Thus there is a loop in the
vertex $S_n$ of the valued Gabriel quiver of $eCe$, $Q_e$. In
addition, $\Socle^2 \overline{E}_0/S_0\cong S_n$, so $Q_e$ contains
the subquiver

$$\xymatrix@R=40pt{ S_n \ar[r] \ar@(ur,ul)[] & S_0}.$$
By Theorem \ref{shapeserial}, $eCe$ is not serial. Hence $m$ must
equal $n$.

 Suppose now that $S_h\cong S_0$
for some $h\neq 1$ and, moreover, it is the smallest such integer.
First, by \cite[Theorem 1.9]{navarro}, there is a path in
$(Q_C,d_C)$ of length $h$ starting and ending at $S_0$, i.e., there
is a cycle in $(Q_C,d_C)$. By Theorem \ref{shapeserial},
$(Q_C,d_C)=\widetilde{\mathbb{A}}_{h}$. By a reasoning similar to
the one done above, we may prove that $S_i\ncong S_j$ for any
$i,j<h$ with $i\neq j$. Therefore it remains to show that
$S_{l+h}\cong S_l$ for any $l>1$. We denote by $M$ the right
comodule $E_0/\Socle^h E_0\leq E_0$. Then, for any $l>0$
$$S_{l+h}\cong \displaystyle\frac{\Socle^{l+h+1}E_0}{\Socle^{l+h}E_0}\cong
\displaystyle\frac{\frac{\Socle^{l+h+1}E_0}{\Socle^h E_0}}{
\frac{\Socle^{l+h}E_0}{\Socle^h E_0}}\overset{\blacktriangle}{\cong}
\displaystyle\frac{
\Socle^{l+1}M}{\Socle^lM}\overset{\blacktriangledown}{\cong}
\displaystyle\frac{\Socle^{l+1}E_0}{ \Socle^{l}E_0}\cong S_l,$$
where in $\blacktriangle$ we use \cite[Lemma 1.4]{navarro} and in
$\blacktriangledown$ we use Lemma \ref{lemilla}.
\end{proof}

\begin{proposition} \label{serialcomputable}
Let $C$ be an indecomposable serial coalgebra. $C$ is Hom-computable
if and only if one of the following conditions holds:
\begin{enumerate}[$(a)$]
\item The right valued Gabriel quiver of $C$ are either $
_{\infty}\mathbb{A}_\infty$, or $\mathbb{A}_\infty$, or
$_{\infty}\mathbb{A}$, or $\mathbb{A}_n$ for some $n\geq 1$.
\item The right valued Gabriel quiver of $C$ is
$\widetilde{\mathbb{A}}_n$ for some $n\geq 1$, and $C$ is finite
dimensional.
\end{enumerate}
\end{proposition}
\begin{proof} Let us assume that $C$ verifies either the condition
$(a)$ or the condition $(b)$. First, if $C$ is finite dimensional
then $\Hom_C(E_i,E_j)$ is finite dimensional for each pair of
indecomposable injective comodules. Now, if the valued Gabriel
quiver of $C$ is either $\mathbb{A}_\infty$, or
$_{\infty}\mathbb{A}_\infty$, or $_{\infty}\mathbb{A}$, or
$\mathbb{A}_n$ for some $n\geq 1$; then, by the Periodicity Theorem,
for each indecomposable injective $E$ and each $i\in I_C$,
$\ell_i(E)$ is one or zero. Then $C$ is Hom-computable.

Conversely, it is enough to prove that if $(Q_C,d_C)=
\widetilde{\mathbb{A}}_n$ for some $n\geq 1$, and $C$ is infinite
dimensional, then $C$ is not Hom-computable. Now, since $C$ is
socle-finite, this is a consequence of \cite[Corollary
2.10]{simson06}.
\end{proof}

Following \cite{simson07}, we say that a finitely copresented
indecomposable comodule $M$ is said to be directing if there is no
chain
$$\xymatrix{M \ar[r]^-{f_1} & M_1 \ar[r]^-{f_2} & \ar@{.}[r] &
\ar[r]^-{f_{n}} & M_n \ar[r]^-{f_{n+1}} & M}$$ where $M_i$ is
finitely copresented and indecomposable for any $i=1,\ldots ,n$ and
$f_j$ is a non-zero non-isomorphism for any $j=1,\ldots, n+1$. A
coalgebra is said to be right (left) representation-directed if each
finitely copresented indecomposable right (left) comodule is
directing. Let us now classify serial representation-directed
coalgebras in terms of their valued Gabriel quiver:

\begin{proposition}\label{representationdirected}
Let $C$ be an indecomposable serial coalgebra. The following
statements are equivalent:
\begin{enumerate}[$(a)$]
\item[$(a)$] $C$ is right representation-directed.
\item[$(a')$] $C$ is left representation-directed.
\item[$(b)$] The right valued Gabriel quiver of $C$ are either $
_{\infty}\mathbb{A}_\infty$, or $\mathbb{A}_\infty$, or
$_{\infty}\mathbb{A}$, or $\mathbb{A}_n$ for some $n\geq 1$.
\item[$(b')$] The left valued Gabriel quiver of $C$ are either $
_{\infty}\mathbb{A}_\infty$, or $\mathbb{A}_\infty$, or $
_{\infty}\mathbb{A}$, or $\mathbb{A}_n$ for some $n\geq 1$.
\end{enumerate}
\end{proposition}
\begin{proof}
Since, by Proposition \ref{Gabvalued}, the left valued Gabriel
quiver of $C$ is the opposite to the right one, it is enough to
prove that $C$ is right representation-directed if and only if
$(Q_C,d_C)$ is one of the above valued quivers.

Let us assume that $(Q_C,d_C)=\widetilde{\mathbb{A}}_n$ for some
$n\geq 1$. We use the following labels in the vertices (we omit the
labels of the arrows):
$$
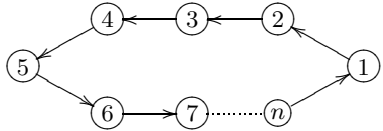
\xymatrix@C=20pt@R=5pt{       & \point{4} \ar[ld]  &
 \point{3} \ar[l] &   \point{2} \ar[l]&    \\
 \point{5} \ar[rd]  &  &    &   & \point{1} \ar[lu]   \\
 & \point{6} \ar[r] &\point{7} \ar@{.}[r] &\point{n} \ar[ru] &  }$$
Then we may consider the chain
$$\xymatrix{S_1\ar[r]^-{inc} & \Socle^2E_1 \ar[r]^-{p} & S_n
\ar[r]^-{inc} & \ar@{.}[r]& \ar[r]^-{p} & S_2 \ar[r]^-{inc} &
\Socle^2 E_2 \ar[r]^-{p} & S_1}$$ where $p$ is the projection of
$\Socle^2 E_i$ onto $\Socle^2 E_i/S_i\cong S_{i-1}$ (or $S_n$ if
$i=1$), and $inc$ is the inclusion of $S_i$ in $\Socle^2E_i$. We
recall that, for each $i=1,\ldots ,n$, $S_i\neq \Socle^2 E_i$ since
$S_i$ is not an isolated point without loops in $Q_C$. Thus the
morphisms $inc$ and $p$ are not isomorphisms. Finally, $\Socle^2
E_i$ is indecomposable and, since $C$ is right serial, finitely
cogenerated. Thus $C$ is not right representation-directed.

Let us now suppose that $Q_C$ is one of the remaining quivers.
By Proposition \ref{serialcomputable}, $C$
is Hom-computable and then, by \cite[Proposition 2.13(c) and Lemma 6.4]{simson07b},
 $C$ is right representation-directed
if and only if  each socle-finite localized coalgebra of $C$ is
right representation-directed. For each finite subset
$\mathfrak{S}\subseteq \{S_i\}_{i\in _C}$, by Proposition
\ref{equivserial},
 the localized coalgebra $D$ associated to
$\mathfrak{S}$ is socle-finite, serial and Hom-computable. Therefore
 its valued Gabriel quiver is either $\mathbb{A}_n$
for some $n\geq 1$, or $\mathbb{\widetilde{A}}_m$ for some $m\geq
1$. Nevertheless, by the Periodicity Theorem, the second case is not
possible and then $(Q_D,d_D)$ is $\mathbb{A}_n$ for some $n\geq 1$.
Then $D$ is finite-dimensional and right representation-directed.
\end{proof}

\section{Finite dimensional comodules over serial
coalgebras}\label{finitedimensional}

This section is devoted to give a complete list of all
indecomposable finite-dimensional right comodules over a serial
coalgebra and a description of the Auslander-Reiten quiver of the
category $\M^C_f$. We recall from \cite{cuadra-serial} that any
finite dimensional indecomposable comodule $M$ over a serial
coalgebra $C$ is uniserial, and then, there exists an integer $t\geq
1$ such that $\Socle^t M=M$. Thus $M$ is a right $D$-comodule, where
$D$ is the subcoalgebra of $C$, $\Socle^t C=\oplus_{i\in I_C}
\Socle^t E_i$ (and then it is serial, cf. \cite{cuadra-serial}). We
refer the reader to \cite{chin2}, \cite{justus}, \cite{simsonnowak}
and \cite{simson1} for definitions and terminology concerning almost
split sequence and the Auslander-Reiten quiver of a coalgebra, see
also \cite{simsonbook1}.

\begin{theorem}\label{fincomod}
Let $C$ be a serial coalgebra. The following statements hold:
\begin{enumerate}[$(a)$]
\item Each finite dimensional indecomposable right $C$-comodule is
isomorphic to $\Socle^k E$ for some positive integer $k$ and some
indecomposable injective right $C$-comodule.
\item The category of finite dimensional right $C$-comodules has
almost split sequences. Furthermore, for each indecomposable
non-injective right $C$-comodule $\Socle^k E$, the almost split
sequence starting on this comodule is
$$\xymatrix{0 \ar[r] & \Socle^k E \ar[r]^-{\binom{i}{p}} & \Socle^{k+1}E\oplus
\displaystyle\frac{\Socle^kE}{\Soc E} \ar[r]^-{(q\hspace{0.05cm}
-j)} & \displaystyle\frac{\Socle^{k+1}E}{\Soc E} \ar[r] & 0},$$
where $i$ and $j$ are the standard inclusions and $p$ and $q$ are
the standard projections.
\end{enumerate}
\end{theorem}
\begin{proof} $(a)$. Let $M$ be a finite dimensional indecomposable comodule,
 by \cite{cuadra-serial}, $M$ is uniserial and then $M=\Socle^tM$
for some $t>0$ (we may consider the minimal one). Since $M$ has
simple socle, its injective envelope is an indecomposable injective
comodule $E$. By Lemma \ref{lemilla}, the Loewy series of $M$ and
$E$ are the same until the step $t$.

$(b)$. Here we essentially follow the proof of \cite[Theorem
4.1]{simsonbook1}. The given sequence is easily seen to be exact.
By Krull-Remak-Schmidt-Azumaya Theorem,
 it is not
split and, since $C$ is right serial, it has indecomposable end
terms. Let us prove that the homomorphism $f=\binom{i}{p}$ is left
almost split. It is clear that f is not a section. Let $N$ be an
indecomposable finite dimensional $C$-comodule, and
$h:\Socle^kE\rightarrow N$ be a non-isomorphism. We have two cases.
If $h$ is not injective, it decomposes through $\Socle^kE/\Soc E$,
namely $h=h'p$. But then the homomorphism $g=(0 \hspace{0.1cm} h')$
satisfies that $g f=h$. If, on the other hand, $h$ is injective,
since it is not an isomorphism, by (a), $N\cong\Socle^t E$ with $t>
k$. Then $N$ is injective as right $\Socle^t C$-comodule and there
exists a morphism $h':\Socle^{k+1} E\rightarrow N$ such that $h=h'
i$. Hence $(h' \hspace{0.1cm} 0)f=h$. Since the left term and the
right term are indecomposable comodules and $f$ is a left almost
split morphism, the sequence is almost split in the category of
finite dimensional $C$-comodules, see \cite[Chapter IV, Theorem
1.13]{simsonbook1}.
\end{proof}

By applying Theorem \ref{fincomod}, we can easily calculate the
Auslander-Reiten quiver of a serial coalgebra. For example, we do it
for hereditary serial path coalgebras.

\vspace{0.2cm}

\noindent \textbf{Type $_{\infty}\mathbb{A}_\infty$}. Let $Q$ be the
quiver
$$\xy \xymatrix@C=30pt@R=10pt{\cdots
\ar[r]^-{\alpha_{-2}} & \point{\text{\tiny -2}} \ar[r]^-{\alpha_{-1}}& \point{\text{\tiny -1}}
\ar[r]^-{\alpha_{0}} & \point{0} \ar[r]^-{\alpha_1} & \point{1} \ar[r]^-{\alpha_2} &
\point{2} \ar[r]^-{\alpha_3}& \cdots }
\endxy$$ and let $C=KQ$ be the path coalgebra of $Q$. By Theorem
\ref{shapeserial}, $C$ is serial. Let $E_i$ be the indecomposable
injective right $C$-comodule associated to the vertex $i$, that is,
$E_i=Ke_i\bigoplus \left (\bigoplus_{t\geq 0}
K\alpha_i\cdots\alpha_{i-t}\right )$, where $e_i$ is the stationary
path at $i$. Let
$$S^k_i=\Socle^kE_i=Ke_i\bigoplus \left (\bigoplus^{k-1}_{t=0}
K\alpha_i\cdots\alpha_{i-t}\right ).$$ Now, since $\Socle^kE_i/\Soc
E_i\cong \Socle^{k-1}E_{i-1}$ for any $k$ and $i$, the almost split
sequences are the following:
$$\xymatrix{0 \ar[r] & S^k_i \ar[r] & S^{k+1}_i\oplus
S^{k-1}_{i-1} \ar[r] & S^k_{i-1} \ar[r] & 0},$$ for each $k\geq 0$
and each $i\in \mathbb{Z}$. Therefore, the Auslander-Reiten quiver
of $C$ is the following.
$$
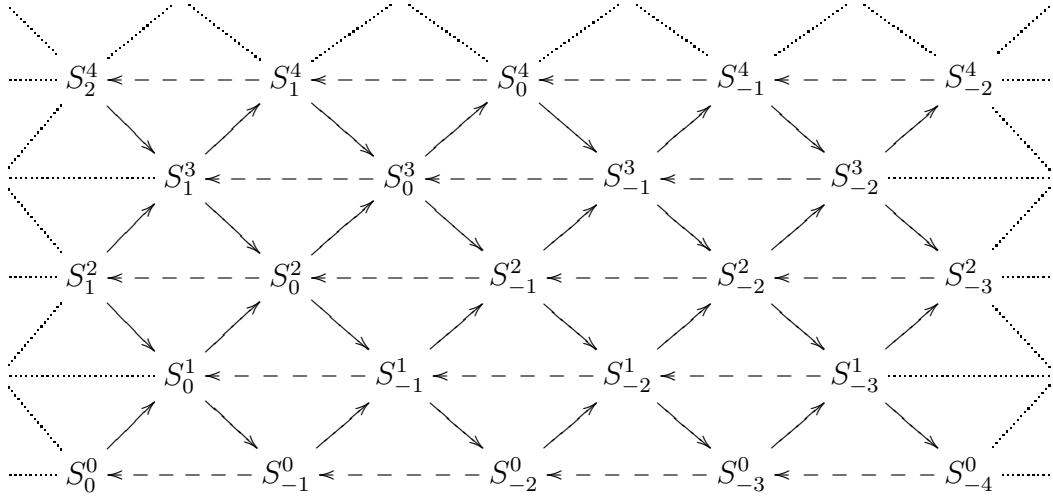
\xymatrix@C=17pt@R=17pt{  \ar@{.}[rd]& & \ar@{.}[rd]  & & \ar@{.}[rd]
  &   & \ar@{.}[rd] &  & \ar@{.}[rd] &  & \\
      \ar@{.}[r]  &  S^4_2 \ar@{.}[ru] \ar[rd] &
      & S^4_1 \ar@{.}[ru] \ar[rd] \ar@{-->}[ll]&
      &  S^4_0  \ar@{.}[ru]\ar[rd]\ar@{-->}[ll]&
      &  S^4_{-1}\ar@{-->}[ll] \ar[rd]\ar@{.}[ru] &
& S^4_{-2} \ar@{.}[ru]\ar@{.}[rd] \ar@{-->}[ll] &
     \ar@{.}[l]          \\
     \ar@{.}[ru]\ar@{.}[rd] \ar@{.}[rr] &     &
       S^3_1 \ar[ru]\ar[rd] &   &
         S^3_0 \ar[ru]\ar[rd]\ar@{-->}[ll] &
             & S^3_{-1} \ar[ru]\ar[rd] \ar@{-->}[ll]&
                 &  S^3_{-2} \ar[ru]\ar[rd]\ar@{-->}[ll] &
& \ar@{.}[ll]\\
     \ar@{.}[r]  &  S^2_1 \ar[ru] \ar[rd]&
     & S^2_0 \ar[ru]\ar[rd] \ar@{-->}[ll]&
     &  S^2_{-1} \ar[ru]\ar[rd]\ar@{-->}[ll] &
        & S^2_{-2} \ar[ru]\ar[rd]\ar@{-->}[ll]  &
    &  S^2_{-3}\ar@{.}[ru]\ar@{.}[rd] \ar@{-->}[ll] &
        \ar@{.}[l]     \\
     \ar@{.}[ru]\ar@{.}[rd]  \ar@{.}[rr]&
     &  S^1_0 \ar[ru] \ar[rd]&
     & S^1_{-1} \ar[ru] \ar[rd]\ar@{-->}[ll] &
         &  S^1_{-2} \ar@{-->}[ll]\ar[ru]\ar[rd]&
             & S^1_{-3}\ar@{-->}[ll] \ar[ru]\ar[rd]  &
& \ar@{.}[ll]  \\
     \ar@{.}[r] &  S^0_0 \ar[ru] &
     & S^0_{-1}  \ar[ru]\ar@{-->}[ll] &
        & S^0_{-2} \ar[ru] \ar@{-->}[ll] &
          & S^0_{-3} \ar[ru]\ar@{-->}[ll]  &
    & S^0_{-4} \ar@{.}[ru] \ar@{-->}[ll] &  \ar@{.}[l]      }$$
where each dashed arrow $\xymatrix{ Y & X\ar@{-->}[l]}$ means that
$Y=\tau (X)$, where $\tau=\text{DTr}$ is the Auslander-Reiten
translation, see \cite{chin2}
 for definitions and details.

\hspace{0.2cm}

\noindent \textbf{Type $_{\infty}\mathbb{A}$}. Let $Q$ be the quiver
$$\xy \xymatrix@C=30pt@R=10pt{\cdots
\ar[r]^-{\alpha_{4}} & \point{4} \ar[r]^-{\alpha_{3}}& \point{3}
\ar[r]^-{\alpha_{2}} & \point{2} \ar[r]^-{\alpha_1} & \point{1} \ar[r]^-{\alpha_0} &
\point{0}}
\endxy$$ and let $C=KQ$ be the path coalgebra of $Q$. Again, by Theorem
\ref{shapeserial}, $C$ is serial. Let $E_i$ be the indecomposable
injective right $C$-comodule associated to the vertex $i$, that is,
$E_i=Ke_i\bigoplus \left (\bigoplus_{t\geq 0}
K\alpha_i\cdots\alpha_{i-t} \right )$, where $e_i$ is the stationary
path at $i$. Let
$$S^k_i=\Socle^kE_i=Ke_i\bigoplus \left (\bigoplus^{k-1}_{t=0}
K\alpha_i\cdots\alpha_{i-t}\right ).$$ Since $\Socle^kE_i/\Soc
E_i\cong \Socle^{k-1}E_{i+1}$ for any $k$ and $i$, the almost split
sequences are the following:
$$\xymatrix{0 \ar[r] & S^k_i \ar[r] & S^{k+1}_i\oplus
S^{k-1}_{i+1} \ar[r] & S^k_{i+1} \ar[r] & 0},$$ for each $i, k\geq
0$. Therefore, the Auslander-Reiten quiver of $C$ is the following.
$$
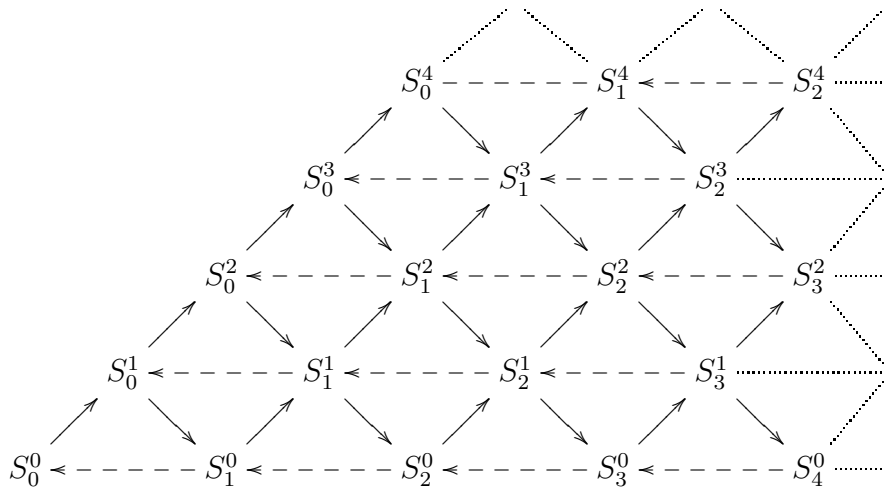
\xymatrix@C=17pt@R=17pt{  & & &
  &   & \ar@{.}[rd] &  & \ar@{.}[rd] &  & \\
      &
      & &
      &  S^4_0  \ar@{.}[ru]\ar[rd]&
      &  S^4_{1}\ar@{--}[ll] \ar[rd]\ar@{.}[ru] &
& S^4_{2} \ar@{.}[ru]\ar@{.}[rd] \ar@{-->}[ll] &
     \ar@{.}[l]          \\
        &
        &   &
         S^3_0 \ar[ru]\ar[rd]&
             & S^3_{1} \ar[ru]\ar[rd] \ar@{-->}[ll]&
                 &  S^3_{2} \ar[ru]\ar[rd]\ar@{-->}[ll] &
& \ar@{.}[ll]\\
       &
     & S^2_0 \ar[ru]\ar[rd] &
     &  S^2_{1} \ar[ru]\ar[rd]\ar@{-->}[ll] &
        & S^2_{2} \ar[ru]\ar[rd]\ar@{-->}[ll]  &
    &  S^2_{3}\ar@{.}[ru]\ar@{.}[rd] \ar@{-->}[ll] &
        \ar@{.}[l]     \\
     &  S^1_0 \ar[ru] \ar[rd]&
     & S^1_{1} \ar[ru] \ar[rd]\ar@{-->}[ll] &
         &  S^1_{2} \ar@{-->}[ll]\ar[ru]\ar[rd]&
             & S^1_{3}\ar@{-->}[ll] \ar[ru]\ar[rd]  &
& \ar@{.}[ll]  \\
       S^0_0 \ar[ru] &
     & S^0_{1}  \ar[ru]\ar@{-->}[ll] &
        & S^0_{2} \ar[ru] \ar@{-->}[ll] &
          & S^0_{3} \ar[ru]\ar@{-->}[ll]  &
    & S^0_{4} \ar@{.}[ru] \ar@{-->}[ll] &  \ar@{.}[l]      }$$

\vspace{0.2cm}

\noindent \textbf{Type $\widetilde{\mathbb{A}}_n$, $n\geq 1$}. Let
$Q$ be the quiver
$$
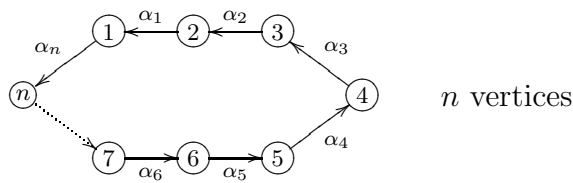
\xy \xymatrix@C=20pt@R=10pt{      & \point{1} \ar[ld]_-{\alpha_n}
 & \point{2} \ar[l]_-{\alpha_1} & \point{3}\ar[l]_-{\alpha_2}  &   &  \\
 \point{n} \ar@{.>}[rd]  &  &    &   & \point{4} \ar[lu]_-{\alpha_3}  & \text{$n$ vertices} \\
  & \point{7} \ar[r]_-{\alpha_6} & \point{6} \ar[r]_-{\alpha_5} & \point{5} \ar[ru]_{\alpha_{4}} &  &
 }

\endxy$$ and let $C=KQ$ be the path coalgebra of $Q$. Clearly, $C$ is serial. Let $E_i$ be the
indecomposable injective right $C$-comodule associated to the vertex
$i$, that is, $E_i=Ke_i\bigoplus \left (\bigoplus_{t\geq 0}
K\alpha_{[i]}\cdots\alpha_{[i-t]} \right )$ where $[p]\equiv p
\text{ (mod $n$)}$ for any $p> 0$. Let
$$S^k_i=\Socle^kE_i=Ke_i\bigoplus \left (\bigoplus^{k-1}_{t=0}
K\alpha_{[i]}\cdots\alpha_{[i-t]}\right ).$$

Now, since $\Socle^kE_n/\Soc E_n\cong \Socle^{k-1}E_{1}$ and
$\Socle^kE_i/\Soc E_i\cong \Socle^{k-1}E_{i+1}$ for each $i=1,\ldots,
n-1$, for any $k\geq 0$, the almost split sequences are the
following:
$$\xymatrix{0 \ar[r] & S^k_i \ar[r] & S^{k+1}_i\oplus
S^{k-1}_{i+1} \ar[r] & S^k_{i+1} \ar[r] & 0},$$ for each
$i=1,\ldots, n-1$ and $k\geq 0$; and
$$\xymatrix{0 \ar[r] & S^k_n \ar[r] & S^{k+1}_n\oplus
S^{k-1}_{1} \ar[r] & S^k_{1} \ar[r] & 0},$$ for any $k\geq 0$.
 Therefore, the Auslander-Reiten
quiver of $C$ is the following.
$$
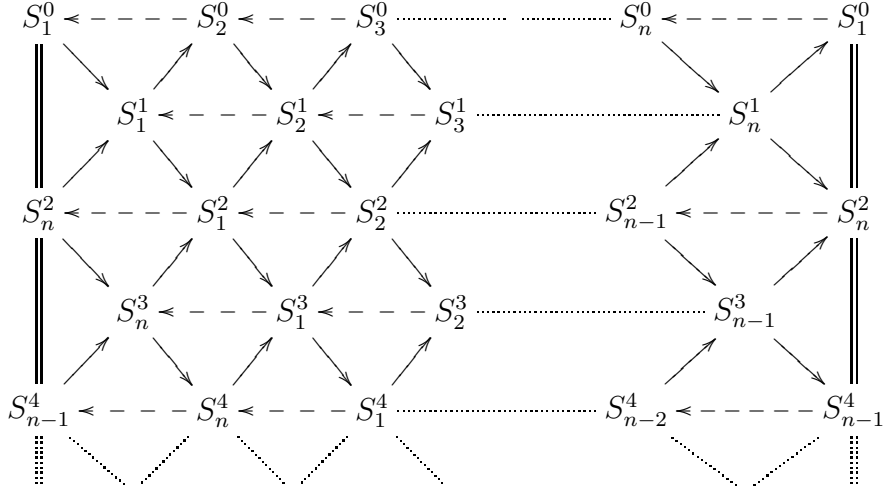
\xymatrix@C=10pt@R=17pt{
          S^0_1 \ar[rd]\ar@{=}[dd] &
      & S^0_2  \ar[rd] \ar@{-->}[ll]&
      & S^0_3  \ar[rd]\ar@{-->}[ll]&
      &  \ar@{.}[ll]  &
& S^0_n \ar[rd] \ar@{.}[ll] & &
     S^0_1   \ar@{-->}[ll]  \ar@{=}[dd]     \\
         &
       S^1_1 \ar[ru]\ar[rd] &   &
         S^1_2 \ar[ru]\ar[rd]\ar@{-->}[ll] &
             & S^1_3  \ar@{-->}[ll]&
                 &    &
&  S^1_n\ar@{.}[llll]  \ar[ru] \ar[rd] &\\
        S^2_n \ar[ru] \ar[rd]\ar@{=}[dd]&
     & S^2_1 \ar[ru]\ar[rd] \ar@{-->}[ll]&
     &  S^2_2 \ar[ru]\ar[rd]\ar@{-->}[ll] &
        &    &
    &  S^2_{n-1}\ar[ru]\ar[rd] \ar@{.}[llll] &
         &   S^2_n \ar@{-->}[ll] \ar@{=}[dd]\\
     &  S^3_n \ar[ru] \ar[rd]&
     & S^3_{1} \ar[ru] \ar[rd]\ar@{-->}[ll] &
         &  S^3_{2} \ar@{-->}[ll]&
             &    &
& S^3_{n-1} \ar@{.}[llll] \ar[ru]\ar[rd] &  \\
       S^4_{n-1} \ar[ru] \ar@{.}[rd] \ar@2{.}[d]&
     & S^4_{n}  \ar[ru]\ar@{-->}[ll] \ar@{.}[rd]&
        & S^4_{1} \ar[ru] \ar@{-->}[ll] \ar@{.}[rd]&
          &  &
    & S^4_{n-2} \ar[ru] \ar@{.}[llll]\ar@{.}[rd] &    & S^4_{n-1} \ar@{-->}[ll] \ar@2{.}[d] \\
 & \ar@{.}[ru] & & \ar@{.}[ru] & & & & & &\ar@{.}[ru]
 & }$$
 This structure is called a stable tube of rank $n$ since the shape
 of the quiver obtained is a tube if we identify the vertical double lines.

\vspace{0.2cm}

The Auslander-Reiten quiver of the remaining serial path coalgebras
are described in \cite{simsonnowak}.

\begin{remark} Observe that, for each indecomposable finite-dimensional comodule
$M=\Socle^nE$, $\tau(M)=\Socle^{k+1} E/\Soc E$, and then $\length M=\length \tau(M)$.
That is, the comodules lying in the same $\tau$-orbit has the same length.
\end{remark}

\section{A theorem of Eisenbud and Griffith for
coalgebras}\label{EisenbudGriffith}

We finish the paper by a version  of the theorem of Eisenbud and
Griffith \cite[Corollary 3.2]{eisenbud1} for coalgebras. We recall
that this theorem asserts that every proper quotient of a hereditary
noetherian prime ring is serial. Obviously, first we need a
translation of the concepts from ring terminology to the notions
used in coalgebra theory. About hereditariness, the concept is
well-known in coalgebras (cf. \cite{blashered}) and it is not needed
any explication. The ``coalgebraic'' version of noetherianess is the
so-called co-noetherianess (cf. \cite{nttstrictly}). We recall that
a comodule $M$ is said to be co-noetherian if every quotient of $M$
is embedded in a finite direct sum of copies
 of $C$. Nevertheless, we shall use a weaker concept: strictly
 quasi-finiteness \cite{nttstrictly}, namely, $M$ is strictly quasi-finite
  if every quotient of $M$ is quasi-finite.
  This is due to fact that we may reduce the problem to socle-finite
   coalgebras and then, under this condition, both classes of comodules coincide
\cite[Proposition 1.6]{nttstrictly}. Finally, following \cite{jmr},
a coalgebra is called prime if for any subcoalgebras $A,B\subseteq
C$
 such that $A\wedge B=C$, then $A=C$ or $B=C$. For the convenience
 of the reader we present the following example:

 \begin{example}\label{example} Let $C$ be a hereditary colocal coalgebra such that $C/S\cong
C\oplus C$, where $S$ is the unique simple comodule (or
subcoalgebra). We prove that $C$ is not co-noetherian.

Let us consider the subcomodule $N_2$ of $C$ which yields the
following commutative diagram
$$\xymatrix@R=15pt{0 \ar[r] & S  \ar@2{=} [d]\ar[r] & N_2 \ar[r] \ar[d]&
S \ar[r]\ar[d]^-{(0,i)} &0 \\
0 \ar[r] & S  \ar[r] & C \ar[r] & C\oplus C \ar[r] &0}$$ Then,
$N_2\subseteq \Socle^2 C$ and $N_2/S\cong S$. Now,
$$C/N_2 \cong
(C/S)/(N_2/S)\cong (C\oplus C)/S\cong (C\oplus C\oplus C)=C^3.$$
Analogously, let $N_3$ be the subcomodule of $C$ which yields the
commutative diagram
$$\xymatrix@R=15pt{0 \ar[r] & N_2 \ar@2{=} [d]\ar[r] & N_3 \ar[r] \ar[d]&
S \ar[r]\ar[d]^-{(0,0,i)} &0 \\
0 \ar[r] & N_2  \ar[r] & C \ar[r] & C\oplus C\oplus C \ar[r] &0}.$$
Again, $N_3/N_2\cong S$ and
$$C/N_3 \cong
(C/N_2)/(N_3/N_2)\cong (C\oplus C\oplus C)/S\cong (C\oplus C\oplus
C\oplus C)=C^4.$$

If we continue in this way, we obtain a increasing family of
subcomodules $\{N_t\}_{t\geq 1}$, where $N_1=S$, such that
$C/N_t\cong C^{t+1}$ for any $t\geq 1$. Let us consider the
uniserial subcomodule $N=\cup_{t\geq 1} N_t$ of $C$ whose
composition series (or Loewy series) is given by
$$ 0\subset S\subset N_2\subset N_3\subset\cdots\subset N.$$
The comodule $C/N$ has infinite dimensional socle. To see this, for
each $i\geq 0$, consider the short exact sequence
$$\xymatrix{0\ar[r] & N/N_i \ar[r] & C/N_i \ar[r] & C/N \ar[r]& 0}$$
which yields the exact sequence
$$\xymatrix{0\ar[r] & \Soc(N/N_i)\ar[r] & \Soc(C/N_i) \ar[r] & \Soc(C/N)},$$
where $\Soc(N/N_i)\cong S$ and $\Soc(C/N_i)\cong S^{i+1}$. Thus
$\dim_K \Soc(C/N)\geq i\cdot\hspace{0.1cm}\dim_K S$ for any $i\geq
1$. Consequently, $C$ is not co-noetherian.
\end{example}

\begin{theorem}\label{theorem} Let $C$ be a basic socle-finite
 coalgebra over an arbitrary field. If $C$ is coprime,
hereditary and co-noetherian then $C$ is serial.
\end{theorem}
\begin{proof} If $C$ is colocal, then the valued Gabriel quiver of $C$
is either a single point or a vertex with a loop labeled by a pair
$(d',d'')$. Now, if $d'\geq 2$ or $d''\geq 2$, proceeding as in
Example \ref{example}, $C$ is not co-noetherian. Thus $d'=d''=1$
and, by Theorem \ref{shapeserial}, $C$ is serial. Assume then that
$C$ is not colocal. Let us first develop some properties about the
valued Gabriel quiver of a localized coalgebra of $C$. These are
inspired by the ones obtained in \cite{jmnr} for path coalgebras.
Let us suppose that $S_x$, $S_y$ and $S_z$ are three simple
$C$-comodules (where $S_x$ could equals $S_z$) such that there is
path in $(Q_C,d_C)$
$$\xymatrix@C=40pt{S_x \ar[r]^-{(d_1,d_2)} & S_y \ar[r]^-{(c_1,c_2)} & S_z}.$$
Let $e\in C^*$ be an idempotent and $eCe$ the localized coalgebra
associated to $e$ whose quotient functor we denote by $T_e$. Assume
that $S_x$ and $S_z$ are torsion-free and $S_y$ is torsion. Since
$C$ is hereditary, $E_z/S_z\cong \bigoplus_{j\in J} E^{r_j}_j
\bigoplus \bigoplus_{t\in T} E^{r_t}_t$, where $S_j$ is torsion for
 all $j\in J$ and $S_t$ is torsion-free for all $t\in T$, and $r_\alpha$
  is a positive integer for any $\alpha\in J\cup T$. Now, since
  there is an arrow from $S_y$ to $S_z$, $y\in J$ and
  $E^{r_y}_y\subseteq E_z/S_z$, where $r_y=c_1$. Then $T_e(E^{c_1}_y)\subseteq T_e(E_z/S_z)$.
Finally, since $S_x^{d_1}\subseteq T_e(E_y)\cong T_e(E_y/S_y)$ then
$S_x^{d_1 c_1}\subseteq T_e(E_z/S_z)=\overline{E}_z/S_z$.
 That is, there exists an arrow
$$\xymatrix@C=40pt{S_x \ar[r]^{(h_1,h_2)} & S_z}$$ in $(Q_{eCe},d_{eCe})$
 such that $h_1\geq d_1 c_1$. By
Proposition \ref{Gabvalued}, it is easy to see that $h_2\geq d_2
c_2$. Note that the hereditariness is a left-right symmetric
property.

By an easy induction one may prove that if there is a path
\begin{equation}\label{path} \xymatrix@C=40pt{S_x \ar[r]^-{(a_0,b_0)} & S_{1}
\ar[r]^-{(a_1,b_1)} & \cdots \ar[r]& S_{n-1}
\ar[r]^-{(a_{n-1},b_{n-1})} & S_n \ar[r]^-{(a_n,b_n)} &
S_z}\end{equation} such that $S_i$ is torsion, for all $i=1,\ldots
,n$. Then there is an arrow
$$\xymatrix@C=40pt{S_x \ar[r]^{(h_1,h_2)} & S_z}$$ in $(Q_{eCe},d_{eCe})$
 such that $h_1\geq a_0a_1\cdots a_n$ and
 $h_2\geq b_0b_1\cdots b_n$. Furthermore, following this procedure, one may prove that
 if $\mathfrak{P}=\{p^l\}_{l\in \Lambda}$ is non empty, where $\mathfrak{P}$ is the set of
 all possible paths $p^l$
 in $(Q_C,d_C)$ as described in (\ref{path}), i.e., starting at $S_x$,
  ending at $S_y$ and whose intermediate vertices are torsion, then
there is an arrow
$$\xymatrix@C=40pt{S_x \ar[r]^{(h_1,h_2)} & S_z}$$ in $(Q_{eCe},d_{eCe})$
 such that $h_1=\sum_{l} a^l_0a^l_1\cdots a^l_{n_l}$ and
 $h_2=\sum_{l} b^l_0b^l_1\cdots b^l_{n_l}$. Here we have denoted by $a^l_0, \ldots ,
 a^l_{n_l}$ and by $b^l_0, \ldots ,
 b^l_{n_l}$
 the first and the second component, respectively, of the labels of the arrows whose
 composition build the path $p^l$. We refer the reader to
 \cite{navarro} for more details about injective comodules and the
 localization functors.

Now we consider a primitive orthogonal idempotent $e_x\in C^*$ and
$e_xCe_x$ the localized coalgebra of $C$ associated to $e_x$. By
\cite[Proposition 1.8]{nttstrictly} and \cite[p. 376, Corollary
5]{gabriel}, $e_xCe_x$ is co-neotherian and hereditary,
respectively. Therefore, following the colocal case, the valued
Gabriel quiver of $e_xCe_x$ must be a single point or a vertex with
a loop labeled by $(1,1)$. As a consequence, by the above
considerations, each vertex of the valued Gabriel quiver of $C$ is
inside of at most one cycle and, if exists, the arrows of that cycle
are labeled by $(1,1)$.

Finally, we prove that for each pair of vertices of $Q_C$ there is a
cycle passing through these two vertices. This yields the statement
of the theorem since, together with the above conditions, the only
possible quiver is $(Q_C,d_C)=\widetilde{\mathbb{A}}_n$ for some
$n\geq 1$ and then $C$ is serial.

Fix two different simple comodules $S_x$ and $S_y$. Let $e_x$ and
$e_y$ be the primitive orthogonal idempotents in $C^*$ associated to
$S_x$ and $S_y$, respectively. We set $e=e_x+e_y$. By
\cite[Proposition 4.1]{jmr}, $eCe$ is prime. First, let us suppose
that there is no path in $Q_C$ from $S_x$ to $S_y$ nor vice versa.
Then $eCe$ has two connected components and, by \cite[Corollary
2.4(b)]{simson06}, $eCe$ is not indecomposable. Thus $eCe$ is not
prime (cf. \cite[Lemma 1.4]{jmr}). Now, suppose that there is a path
from $S_x$ to $S_y$ but there is no path from $S_y$ to $S_x$. Then
the valued Gabriel quiver of $eCe$ is a subquiver of the following
quiver:
$$\xymatrix@R=50pt{ S_x \ar[r]^-{(a,b)} \ar@(ur,ul)[] & S_y \ar@(ur,ul)[]}.$$
By \cite[Lemma 3.7]{navarro}, $e_xCe_y=T_x(E_y)\neq 0$ and
$e_yCe_x=T_y(E_x)=0$. Therefore, there is a vector space direct sum
decomposition $eCe=e_xCe_x\oplus e_xCe_y\oplus e_yCe_y$. A
straightforward calculation shows that the linear map
$$\Psi: D=\left(
               \begin{array}{cc}
                 e_yCe_y & e_xCe_y \\
                 0 & e_xCe_x \\
               \end{array}
             \right)  \longrightarrow eCe,$$
           between $eCe$ and the bipartite coalgebra $D$ (in the sense of \cite{justus2}),
             given by $$\Psi\left(
                                    \begin{array}{cc}
                                      a & b \\
                                      0 & c \\
                                    \end{array}
                                  \right) =a+b+c$$ is an
                                  isomorphism of coalgebras, see \cite{cuadra} for definitions
                                  and details about the structures of the spaces
                                  $e_xCe_x$, $e_yCe_y$ and $e_xCe_y$. We recall that the coalgebra structure of the
                                  bipartite coalgebra $D$ is given by the formulae:

\begin{itemize}
\item
$\Delta(a+b+c)=\Delta_{y}(a)+\rho_y(b)+\rho_x(b)+\Delta_{x}(c)$,
where $\rho_y$ and $\rho_x$ are the $e_yCe_y$-$e_xCe_x$-bicomodule
structure maps of $e_xCe_y$; and $\Delta_x$ and $\Delta_y$ are the
comultiplication of the coalgebras $e_yCe_y$ and $e_xCe_x$,
respectively.
\item $\epsilon(a+b+c)=\epsilon_y(a)+\epsilon_x(b)$, where
$\epsilon_y$ and $\epsilon_x$ are the counit of the coalgebras
$e_yCe_y$ and $e_xCe_x$, respectively.
\end{itemize}
                                  Then
                                   $eCe=e_yCe_y\wedge e_xCe_x$ and therefore $eCe$ is not prime.
 \end{proof}

Now we prove the Eisenbud-Griffith Theorem for coalgebras.

\begin{corollary}\label{EGtheorem}
If $C$ is a subcoalgebra of a prime, hereditary and strictly
quasi-finite (left and right)
 coalgebra over an arbitrary field,
then $C$ is serial.
\end{corollary}
\begin{proof}
By \cite[Proposition 1.5]{cuadra-serial}, we may assume that $C$ is
prime, hereditary and strictly quasi-finite itself. Let $e\in C^*$
be an idempotent such that the localized coalgebra
 $eCe$ is socle-finite.
By \cite[p. 376, Corollary 5]{gabriel}, \cite[Proposition 4.1]{jmr}
and \cite[Proposition 1.8]{nttstrictly},
 $eCe$ is hereditary, prime
and right strictly quasi-finite, respectively. Moreover, by
\cite[Proposition 1.5]{nttstrictly},
 $eCe$ is co-noetherian. Therefore,
by Theorem \ref{theorem}, $eCe$ is serial. Thus the result follows
from Proposition \ref{equivserial}.
\end{proof}

\end{document}